\documentclass[reqno, 11pt]{amsart}
\usepackage[utf8]{inputenc}
\usepackage[english]{babel}
\usepackage[pdftex,usenames,svgnames,dvipsnames]{xcolor}
\usepackage[a4paper,top=3.5cm,bottom=3cm,left=3.5cm,right=3.5cm,marginparwidth=1.75cm]{geometry}
\usepackage{amsmath,amsthm,amssymb,amsfonts}
\usepackage{mathrsfs, mathtools}
\usepackage{import}
\usepackage{enumitem}
\usepackage{xifthen}
\usepackage{pdfpages}
\usepackage{transparent}
\usepackage{float}
\usepackage{comment}
\usepackage{subfigure}
\usepackage{hyperref}
\usepackage{tikz-cd}
\usepackage{stackrel}
\usepackage{wrapfig}
\usepackage{amsaddr}
\usepackage[misc]{ifsym}
\usepackage[pagewise]{lineno}
\newcommand\restr[2]{{ %comando pra restrição de função ficar bonita
  \left.\kern-\nulldelimiterspace 
  #1 
  \vphantom{\big|} 
  \right|_{#2} 
  }}
  
\theoremstyle{plain}
\newtheorem{theorem}{Theorem}%[chapter]
\newtheorem{lemma}{Lemma}%[chapter]

\newtheorem{proposition}{Proposition}%[chapter]

\theoremstyle{definition}

\newtheorem{definition}{Definition}%[chapter]

\newtheorem{remark}{Remark}%[chapter]

\newtheorem{manualtheoreminner}{Theorem}
\newenvironment{manualtheorem}[1]{%
  \renewcommand\themanualtheoreminner{#1}%
  \manualtheoreminner
}{\endmanualtheoreminner}

\title[Equilibrium states for partially hyperbolic maps with 1-d center]{Equilibrium states for partially hyperbolic maps with one-dimensional center}

\date{\today}

\author{Carlos F. Álvarez}
\address{Universidad del Sinú Seccional Cartagena, Área de Ciencias Básicas Exactas,\\ Av. El Bosque, Trasnversal 54 Nº 30-72, Cartagena de Indias 130001, Colombia}
\email{carlosfalvarez@unisinu.edu.co}

\author{Marisa Cantarino}
\address{Universidade Federal Fluminense, Instituto de Matemática e Estatística, \\ Rua Professor Marcos Waldemar de Freitas Reis, s/n - Campus do Gragoata\\ Niterói 24210-201, Brazil}
\email{\Letter mcantarino@id.uff.br}

\thanks{\textbf{Acknowledgments:} CF. A. thanks Universidad del Sinú, seccional Cartagena for the opportunity to develop this work. M. C. was partially financed by the Coordenação de Aperfeiçoamento de Pessoal de Nível Superior - Brasil (CAPES) and Fundação Carlos Chagas Filho de Amparo à Pesquisa of the State of Rio de Janeiro (FAPERJ)}

\keywords{Equilibrium states; measures of maximal entropy; partially hyperbolic endomorphisms; intrinsic ergodicity}
\subjclass[2020]{Primary: 37D35; Secondary: 37D30}

\begin{document}
 
\begin{abstract}

We prove the existence of equilibrium states for partially hyperbolic endomorphisms with one-dimensional center bundle. We also prove, regarding a class of potentials, the uniqueness of such measures for endomorphisms defined on the $2$-torus that: have a linear model as a factor; and with the condition that this measure gives zero weight to the set where the conjugacy with the linear model fails to be invertible.  In particular, we obtain the uniqueness of the measure of maximal entropy. For the $n$-torus, the uniqueness in the case with one-dimensional center holds for absolutely partially hyperbolic maps with additional hypotheses on the invariant leaves, namely, dynamical coherence and quasi-isometry.

%We prove the existence of measures of maximal entropy for partially hyperbolic endomorphisms with one-dimensional center bundle. We also prove the uniqueness of such measures for endomorphisms defined on the $2$-torus that have a linear model as a factor. For the $n$-torus, the uniqueness in the case with one-dimensional center holds with additional hypotheses on the invariant leaves, namely, dynamical coherence and quasi-isometry.
\end{abstract}
\maketitle

\section{Introduction}

Ergodic theory is the branch of mathematics that studies dynamical systems equipped with an invariant probability measure. This theory was motivated by statistical mechanics, aiming, for instance, to understand and solve problems connected with the kinetic theory of gases, such as whether a Hamiltonian system is ergodic or not. A system is called \textit{ergodic} when it is dynamically indivisible from a measure theoretical point of view, meaning that each invariant set has either zero or total measure. 

There are two kinds of entropy used to describe the complexity of a dynamical system: metric entropy and topological entropy. The \textit{metric entropy} provides the maximum amount of average information one can obtain from a system with respect to an invariant measure, and the \textit{topological entropy} describes the exponential growth rate of the number of orbits. 

The variational principle establishes that the topological entropy coincides with the supremum of the metric entropy over all invariant measures. A \textit{measure of maximal entropy} is an invariant measure with its metric entropy coinciding with the topological entropy of the system, in other words, it is a measure in which the supremum is achieved. Such measures ``perceive'' the whole complexity of the system. %, and they are part of the study of thermodynamic formalism, see \cite[Chapter~10]{viana_oliveira_2016} and the references therein. 

A generalization of topological entropy is the \textit{pressure} of the map $f: X \to X$ with respect to a continuous potential $\phi: X \to \mathbb{R}$, which is a ``weighted entropy'' according to $\phi$. This quantity also obeys a variational principle, stating that the pressure coincides with the supremum of the metric entropy plus $\int \phi d\mu$ over all invariant measures $\mu$. An invariant measure for $f$ is called an \textit{equilibrium state} if it achieves this supremum. This measure maximizes the free energy of the system. In particular, an equilibrium state with respect to the potential $\phi \equiv 0$ is called \textit{measure of maximal entropy.}

Describing the existence and uniqueness/finiteness of equilibrium states is an active research topic in dynamical systems. A particular case of interest is when a system has a unique measure of maximal entropy, and we say that it is \textit{intrinsically ergodic}. Uniformly hyperbolic diffeomorphisms are intrinsically ergodic \cite{bowen,margulis}, while partially hyperbolic diffeomorphisms can be intrinsically ergodic \cite{buzzifishersambarinovasquez2012} or not, as illustrated by a simple example. Indeed, consider $f = A \times Id: \mathbb{T}^n \times \mathbb{S}^1 \to \mathbb{T}^n \times \mathbb{S}^1$, where $A$ is an Anosov toral automorphism. Then $A$ has only one measure of maximal entropy, the volume $m$ on $\mathbb{T}^n$, and $m \times \mu$ has the same metric entropy as $m$ for any probability measure $\mu$ on $\mathbb{S}^1$. This product measure is $f$-invariant and the isometric factor on $\mathbb{S}^1$ does not contribute to the topological entropy, so any such product is a measure of maximal entropy for $f$.

Although the above example works for $A$ an Anosov toral endomorphism, for non-invertible systems with some kind of hyperbolicity this problem has been much less explored in its generality.

A \textit{uniformly hyperbolic endomorphism} (also called \textit{Anosov endomorphism)} is a generalization of the concepts of uniformly hyperbolic diffeomorphism and expanding map, being hyperbolic but not necessarily invertible. Anosov endomorphisms in general are not structurally stable, since any two nearby maps may have different amounts of unstable manifolds for corresponding points, which obstructs the existence of a conjugacy \cite{mane1975stability, przytycki1976anosov}. Nonetheless, their dynamics can be studied in spaces of orbits where a kind of stability holds, and it can be shown in this way that they are intrinsically ergodic, see Section \ref{sec:preli} for more details.

Partially hyperbolic endomorphisms are the non-invertible generalization of partial hyperbolicity, as we define on Section \ref{sec:preli}. There are examples of partially hyperbolic endomorphisms that are not intrinsically ergodic on the product $\mathbb{T}^n \times I$, having an expanding map as a factor on $\mathbb{T}^n$ \cite{kan1994, nunezramirezvasquez}. The example $f \times Id: \mathbb{T}^n \times \mathbb{S}^1 \to \mathbb{T}^n \times \mathbb{S}^1$, with $f$ an Anosov endomorphism, has non-hyperbolic linearization and infinitely many measures of maximal entropy. But the intrinsic ergodicity for partially hyperbolic endomorphisms with hyperbolic linearization presenting contracting directions was not yet explored, to the best of our knowledge.

Our first effort is to guarantee a general result of existence of equilibrium states in the case that the central direction is one-dimensional, as the result by W. Cowieson and L.-S. Young in \cite{cowiesonyoung} for the invertible case.

\begin{manualtheorem}{A}
    \label{teoA}
    For $M$ a closed Riemannian manifold, if $f: M \to M$ is a $C^1$ partially hyperbolic endomorphism with one-dimensional center bundle, then there is an equilibrium state for $(f, \phi)$, for any continuous potential $\phi: M \to \mathbb{R}$.
\end{manualtheorem}

Observe that in the above result we do not require $f$ to be absolutely partially hyperbolic (see Definition \ref{def:abs-ph}), since we do not use any finer estimate of the derivatives. Neither we require $f$ to be dynamically coherent (see Definition \ref{def:dyn-cohe}), because the central direction is one dimensional, which implies that the central bundle (we can define a bundle over connected components of the natural extension space, see Subsection \ref{sebsec:inverse-limit}) is integrable, even if it is not necessarily uniquely integrable.

We then explore a case in the torus for which uniqueness of the measure of maximal entropy can be established, with an approach inspired by \cite{buzzifishersambarinovasquez2012, ures2012intrinsic}. %We use a semiconjugacy with an expansive map on the space of orbits to obtain a local result in the form of Theorem \ref{teoB}. 

%\begin{manualtheorem}{B}
    %\label{teoB}
 % If $f: \mathbb{T}^n \to \mathbb{T}^n$ is a dynamically coherent $C^1$ absolutely partially hyperbolic endomorphism with one-dimensional center bundle, its linearization $A$ is hyperbolic and $d_{C^1}(f,A)$ is sufficiently small, then $f$ has a unique measure of maximal entropy $\mu$ and $(f, \mu)$ is ergodically equivalent to $(A, m)$, where $m$ is the volume measure on $\mathbb{T}^n$.
%\end{manualtheorem}

%If there is a semiconjugacy between $f$ and a hyperbolic linear endomorphism $A$ on $\mathbb{T}^n$, we get uniqueness with further hypotheses.

\begin{manualtheorem}{B}
    \label{teoB}
    Let $f: \mathbb{T}^n \to \mathbb{T}^n$ be an absolutely partially hyperbolic endomorphism that is dynamically coherent, with $\dim (E^c) = 1$, and with $W^\sigma_F$ (for $\sigma \in \{u,c,s\}$) quasi-isometric, where $F: \mathbb{R}^n \to \mathbb{R}^n$ is a lift of $f$. If $A = f_*$ is hyperbolic and $A$ is a factor of $f$, then $f$ has a unique measure of maximal entropy $\mu$. Moreover, $(f, \mu)$ and $(A, m)$ are ergodically equivalent.
\end{manualtheorem}

%The previous result holds on $\mathbb{T}^n$ if the center direction is one-dimensional, $f$ is dynamically coherent and it has quasi-isometric leaves $W^\sigma_f$, where $\sigma \in \{u,c,s\}$, as we see in Theorem \ref{teoD} on Section \ref{sec:teoB}.

In the invertible setting, the work of M. Brin \cite{brin2003dynamical} gives us that the quasi-isometry of strong foliations guarantees dynamical coherence for absolutely partially hyperbolic diffeomorphisms. Even if this is true on our non-invertible context, we still need absolute partial hyperbolicity in the proof of Lemma \ref{lem:lemma1}. L. Hall and A. Hammerlindl have proved that partially hyperbolic surface endomorphisms with hyperbolic linearization are dynamically coherent and their foliations by unstable and central leaves in the universal cover are quasi-isometric \cite{hall2021partially}. Then, Theorem \ref{teoB} has a simpler formulation on $\mathbb{T}^2$. Observe that in this case the hypothesis on absolute partial hyperbolicity can be dropped, since Lemma \ref{lem:lemma1} is trivially valid for $A: \mathbb{T}^2 \to \mathbb{T}^2$.

\begin{manualtheorem}{C}
\label{teoC}
Let $f: \mathbb{T}^2 \to \mathbb{T}^2$ be a partially hyperbolic endomorphism with hyperbolic linearization $A$. If $A$ is a factor of $f$, then $f$ has a unique measure of maximal entropy $\mu$, and $(f, \mu)$ is ergodically equivalent to $(A, m)$.      
\end{manualtheorem}

Let $f: \mathbb{T}^n \to \mathbb{T}^n$ be a partially hyperbolic endomorphism with hyperbolic linearization $A$. Then a lift $F: \mathbb{R}^n \to \mathbb{R}^n$ of $f$ to the universal cover has $A: \mathbb{R}^n \to \mathbb{R}^n$ as a factor --- where $A$ is the unique linear map that projects to $A: \mathbb{T}^n \to \mathbb{T}^n$, and we use the same notation for both maps. In other words, there is a continuous and surjective map $H: \mathbb{R}^n \to \mathbb{R}^n$ such that $H \circ F = A \circ H$. It remains to explore conditions for this semiconjugacy to descend to $\mathbb{T}^n$. In the uniformly hyperbolic case, $H$ is a conjugacy, and it projects to $\mathbb{T}^n$ if and only if $f$ has a unique unstable manifold for each point \cite{sumi1994linearization}. But there is no similar result for the partial hyperbolic case.

A more general result can be obtained in Theorem \ref{teoD} for a specific class of potentials, the ones that are written as $\phi \circ h$, with $\phi: \mathbb{T}^n \to \mathbb{R}$ continuous and $h: \mathbb{T}^n \to \mathbb{T}^n$ being the semiconjugacy that makes $A$ a factor of $f$. For this, we require that the corresponding equilibrium state $\mu$ for $(A, \phi)$ gives measure $0$ to $\{ z \in \mathbb{T}^n \; : \; \# h^{-1}(z) > 1 \}$, as in \cite[Theorem A]{crisostomo-tahzibi2019}.

\begin{manualtheorem}{D}
    \label{teoD}
    Let $f: \mathbb{T}^n \to \mathbb{T}^n$ be an absolutely partially hyperbolic endomorphism that is dynamically coherent, with $\dim (E^c) = 1$, and with $W^\sigma_F$ (for $\sigma \in \{u,c,s\}$) quasi-isometric, where $F: \mathbb{R}^n \to \mathbb{R}^n$ is a lift of $f$. Consider a continuous potential $\phi: \mathbb{T}^n \to \mathbb{R}$ and $\mu$ the unique equilibrium state for $(A, \phi)$. If $A = f_*$ is hyperbolic, $A$ is a factor of $f$ with semiconjugacy $h$, and
    $$\mu\{ z \in \mathbb{T}^n \; : \; \# h^{-1}(z) > 1 \} = 0,$$
    then $f$ has a unique equilibrium state $\hat{\mu}$ for $(f, \phi \circ h)$. Moreover, $(f, \hat{\mu})$ and $(A, \mu)$ are ergodically equivalent.
\end{manualtheorem}

As before, the above result holds for partially hyperbolic endomorphisms with hyperbolic linearization in $\mathbb{T}^2$, without requiring the additional hypotheses that we use on the higher dimensional setting. The proof of Theorem \ref{teoD} runs exactly as the one for Theorem \ref{teoB} after Lemma \ref{lem:lemma3}, which is replaced by the hypothesis on $\{ z \in \mathbb{T}^n \; : \; \# h^{-1}(z) > 1 \}$.

Further results can be explored by the investigation of conditions under which we can obtain dynamical coherence and quasi-isometry of foliations for general partially hyperbolic endomorphisms. Another possibility would be to consider the case in which $\dim (E^c) > 1$.

\subsection*{Organization of the paper}

Section \ref{sec:preli} is dedicated to summarize concepts that give context and tools for our results. In Section \ref{sec:teoA} we give the proof of the existence of equilibrium states by proving that, under the hypotheses of Theorem \ref{teoA}, $f$ is $h$-expansive. We address the proof of uniqueness on Section \ref{sec:teoB}.

%%%    PRELIMINARES   %%%
\section{Preliminary concepts}
\label{sec:preli}

In this section we introduce some necessary concepts and results for our proofs. Here $M$ is an closed smooth manifold (compact, connected and without boundary), and $(X,d)$ is a compact metric space.

\subsection{Equilibrium states}

We now review a few facts on entropy and equilibrium states. We start this section with an alternative definition of metric entropy.
\vspace{0.1cm}

Let $f:X\to X$ be a continuous map and $\mu$ an ergodic $f$-invariant probability measure. We define the $d_{n}$ metric over $X$ as $$d_{n}(x,y):=\displaystyle\max_{0\leq i\leq n-1}d(f^{i}(x), f^{i}(y)).$$

For $\delta\in (0,1)$, $n\in \mathbb{N}$ and $\epsilon>0$, a finite set $E\subset X$ is called $(n,\epsilon,\delta)$-\textit{spanning} if the union of the $\epsilon$-balls $B^{n}_{\epsilon}(x)=\{y\in X: d_{n}(x,y)<\epsilon\}$, centered at points $x\in E$, has $\mu$-measure greater than $1-\delta$. The \textit{metric entropy} is defined by
$$h_{\mu}(f)=\displaystyle\lim_{\epsilon \rightarrow 0}\displaystyle\limsup_{n\rightarrow \infty}\frac{1}{n}\log \left( \min\{\# E: E\subseteq X \ {\rm is \ }(n,\epsilon,\delta)-\rm spanning \} \right).$$

Let $K\subset X$ be a non-empty compact set. A set $E\subseteq K$ is said to be $(n,\epsilon)$-\textit{spanning} if $K$ is covered by the union of the $\epsilon$-balls centered at points of $E$. The \textit{topological entropy} of $f$ on $K$ is defined by

$$h(f,K)=\displaystyle\lim_{\epsilon \rightarrow 0}\displaystyle\limsup_{n\rightarrow \infty}\frac{1}{n}\log \left( \min\{\# E: E\subseteq K \ \mbox{\rm{is}} \ (n,\epsilon)-\mbox{\rm{spanning}} \}\right).$$
We denote $h_{top}(f):=h(f,X).$

The classical definition of metric entropy can be seen, for instance, on \cite[Chapter 9]{viana_oliveira_2016}, and the above characterization is for ergodic measures on compact metric spaces.

Dinaburg's variational principle establishes a relation between metric entropy and topological entropy. It states that $$\sup \{h_{\mu}(f):\mu\in \mathcal{M}(X,f)\}=\sup \{h_{\mu}(f):\mu\in \mathcal{M}_{e}(X,f)\}=h_{top}(f),$$ where $\mathcal{M}(X,f)$ denotes the set of $f$-invariant Borel probability measures on $X$ and $\mathcal{M}_{e}(X,f)$ denotes the set of ergodic measures on $X$.

\begin{definition}
A \textit{measure of maximal entropy} is a probability measure $\mu\in \mathcal{M}(X,f)$ such that $h_{\mu}(f)=h_{top}(f).$   
\end{definition}

A generalization of the concept of entropy is given by \textit{(topological) pressure}. We define the \textit{pressure} of $f$ with respect to the continuous \textit{potential} $\phi: M \to \mathbb{R}$ by 
$$P(f,\phi)=\displaystyle\lim_{\epsilon \rightarrow 0}\displaystyle\limsup_{n\rightarrow \infty}\frac{1}{n}\log \left( \inf \left\{\sum_{x \in E}e^{\phi_n(x)}: E\subseteq M \ \mbox{\rm{is}} \ (n,\epsilon)-\mbox{\rm{spanning}} \right\}\right),$$
where $\phi_n(x) = \sum_{i = 0}^{n-1} \phi \circ f^i$. And Walter's variational principle gives us that
$$P(f,\phi) = \sup_{\mu \in \mathcal{M}(X,f)} \left\{h_{\mu}(f) + \int_M \phi d\mu \right\}.$$

\begin{definition}
An \textit{equilibrium state} for the pair $(f, \phi)$ is a probability measure $\mu\in \mathcal{M}(X,f)$ such that $P(f,\phi) = h_{\mu}(f) + \int_M \phi d\mu.$  
\end{definition}

In general, establishing the existence of equilibrium states is a non-trivial task, but in certain contexts there are results that facilitate this task. For instance, when the pressure function $P_\phi:\mathcal{M}(X,f)\to [0,\infty)$ defined by $P_\phi(\mu):= h_{\mu}(f) + \int_M \phi d\mu$ is upper semi-continuous, $f$ admits an equilibrium state with respect to $\phi$.

Let $f:X\to X$ be a homeomorphism. The \textit{bi-infinite Bowen ball} around $x\in X$ of size $\epsilon>0$ is the set $$\Gamma_{\epsilon}(x):=\{y\in X: d(f^{n}(x),f^{n}(y))<\epsilon \ \mbox{\rm{for all}} \ n\in \mathbb{Z}\}.$$

We say that $f$ is \textit{expansive} if there is a constant $\epsilon>0$ such that $\Gamma_{\epsilon}(x)=\{x\}$ for all $x\in X.$ 

To generalize the concept of expansiveness, consider %the ball $B^n_\epsilon(x) = \{y \in X : d_{n}(x,y)<\epsilon \}$ and let 
$$\phi_\epsilon(x) := \bigcap_{n =1}^\infty B^n_\epsilon(x).$$

\begin{definition}%($h$-expansive)\\
A continuous map $f:X\to X$ is called:
\begin{enumerate}
    \item \textit{$h$-expansive} if $h^{\ast}_{f}(\epsilon):=\sup_{x\in M}h(f,\phi_{\epsilon}(x))=0;$
    \item \textit{asymptotically $h$-expansive} if  $\displaystyle\lim_{\epsilon\to 0}h^{\ast}_{f}(\epsilon)=0.$
\end{enumerate}
\end{definition}

Of course $h$-expansiveness implies asymptotic $h$-expansiveness. It is also easy to check that, if $f$ is a homeomorphism, then $\Gamma_\epsilon(x) \subseteq \phi_\epsilon(x)$ and by \cite[Corollary 2.3]{bowen-h-expansive1972}, $h^{\ast}_{f}(\epsilon) = h^{\ast}_{f, homeo}(\epsilon) :=\sup_{x\in M}h(f,\Gamma_{\epsilon}(x))$. In particular, if $f$ is expansive, then it is $h$-expansive.

\begin{theorem}[\cite{misiurewicz1973diffeomorphism}]\label{emme}
If $f:X\to X$ is asymptotically $h$-expansive, then the entropy function is upper semi-continuous. In particular, $f$ admits an equilibrium state for any continuous potential. 
\end{theorem}

Uniqueness of equilibrium states is a much more delicate problem, but it also brings interesting properties. For instance, when such measures are unique, they are ergodic. Some classes of maps are known to have a unique equilibrium state. For expansive homeomorphisms with specification property, for instance, this problem was solved for Hölder continuous potentials by R. Bowen in \cite{bowen}.

Following B. Weiss \cite{weiss1970}, we call a system \textit{intrinsically ergodic} if there is a unique measure of maximal entropy. Our main tool to prove intrinsic ergodicity is the following theorem, known as Ledrappier--Walters' formula.

%Let $Y$ be a compact metric space. A dynamical system $(X, f)$ has a \textit{symbolic extension} if there exists a subshift $(Y,\sigma)$ and a continuous surjective map $\pi:Y \to X$ such that $\pi \circ \sigma= f\circ \sigma$; the system $(Y, \sigma)$ is called an \textit{extension} of $(X, f)$ and $(X, f)$ is called a factor of $(Y,\sigma)$.

%\begin{theorem}[\cite{boyle-friebig}]\label{Boyle}
%If $f: X \to X$ is asymptotically $h$-expansive, then $f$ has a principal symbolic extension.
%\end{theorem}

%Partial hyperbolicity has been extensively studied as a tool to understand dynamical and ergodic properties of diffeomorphisms. Compared to diffeomorphisms, some ergodic properties for endomorphisms are less understood.
%In this paper, we consider a family of partially hyperbolic endomorphisms over $\mathbb{T}^n$, and we show that these systems are intrinsically ergodic.
%There are some difficulties to implement many arguments holds for diffeomorphisms. For instance, the preimage of a given point is usually not a single point, so the unstable manifolds may not be well defined. Moreover, not always it is clear the existence of a semiconjugacy between a partially hyperbolic endomorphism and its linearization.

\begin{theorem}{\cite[Theorem 2.1]{ledrappier1977relativised}}\label{principiolw}
   Let $X$ and $Y$ be compact metric spaces and consider $T: X \to X$, $S: Y \to Y$ and $h: X \to Y$ continuous maps with $h$ surjective and such that $S \circ h = h \circ T$. If $\phi: X \to X$ is continuous and $\nu$ is an $S$-invariant probability measure, then      $$\sup\{h_\mu(T | S) + \int \phi d \mu \; : \; \mu \in \mathcal{M}(X, T) \mbox{ and } \mu \circ \pi^{-1} = \nu \} = \int_Y P(T,\phi,h^{-1}(y)) d\nu(y).$$     In particular, if $\phi=0$, we have that
   $$\sup\{h_\mu(T) \; : \; \mu \in \mathcal{M}(X, T) \mbox{ and } \mu \circ h^{-1} = \nu \} = h_\nu(S) + \int_Y h_{top}(T,h^{-1}(y)) d\nu(y).$$
\end{theorem}

\subsection{Natural extensions}
\label{sebsec:inverse-limit}

For certain aspects, we need to analyze the past orbit of a point. If the map is not invertible, every point has more than one preimage, and there are several ``choices of past''. We can make each one of these choices a point on a new space, defined as follows.

\begin{definition}
Given $f: X \to X$ continuous, the \emph{natural extension} (or \emph{inverse limit space}) associated to the triple $X$, $d$ and $f$ is
\begin{itemize}
    \item $X_f = \{\Tilde{x} = (x_k) \in X^\mathbb{Z}: x_{k+1} = f(x_k) \mbox{, } \forall k \in \mathbb{Z}\}$,
    \item $(\Tilde{f}(\Tilde{x}))_k = x_{k+1}$ $\forall k \in \mathbb{Z}$ and $\forall \Tilde{x} \in X_f$, 
    \item $\Tilde{d}(\Tilde{x}, \Tilde{y}) = \sum\limits_k \dfrac{d(x_k, y_k)}{2^{|k|}}$.
\end{itemize}
\end{definition}

We also denote $(X_f, \tilde{f})$ as $\varprojlim(X,f)$. We have that $(X_f,\Tilde{d})$ is a compact metric space and the shift map $\Tilde{f}$ is continuous and invertible. Let $\pi: X_f \to X$ be the projection on the 0th coordinate, $\pi(\tilde{x}) = x_0$, then $\pi$ is a continuous surjection and $f \circ \pi = \pi \circ \tilde{f}$. Therefore, every non-invertible topological dynamical system on a compact metric space is a topological factor of an invertible topological dynamical system on a compact metric space.

With the metric $\tilde{d}$ over $X_f$, we can define precisely the continuity of objects that depend on the orbit of a point, such as the invariant manifolds given by the invariant splittings in the definitions of the next subsection. 

Even when $X = M$ is a manifold, the space of natural extension does not have a manifold structure, but we can still pull back the tangent structure of the original manifold to the natural extension. By doing so, we have a manifold structure on each connected component of $M_f$, which is diffeomorphic to the universal covering $\Tilde{M}$ of $M$. On each one of these connected components, we have unstable and stable foliations. Additionally, if $\dim E^c = 1$ there are central curves tangent to $E^c$ restricted to this space.

By making use of the inverse limit space, it is also possible to better comprehend invariant measures for non-invertible systems. Unless stated otherwise, all measures on this work are over the Borel $\sigma$-algebra on the given space. For any $\tilde{f}$-invariant probability measure $\tilde{\mu}$ on $X_f$, we obtain $\pi_* \tilde{\mu}$ as a probability measure on $X$ that is $f$-invariant. What makes the natural extension the natural construction to study ergodic theory of non-invertible systems is the fact that $\pi_*$ is actually a bijection between the invariant probabilities for $(X_f, \tilde{f})$ and $(X,f)$, as stated in \cite[Proposition I.3.1]{qian2009smooth}.

\begin{proposition}[\cite{qian2009smooth}]
    \label{prop:corresp-measures}
    Let $(X,d)$ be a compact metric space and $f: X \to X$ continuous. For any $f$-invariant probability measure $\mu$ on $X$, there is a unique $\tilde{f}$-invariant probability measure $\tilde{\mu}$ on $X_f$ such that $\pi_* \tilde{\mu} = \mu$.
\end{proposition}

Moreover, we have that the metric entropies for corresponding measures are the same \cite[Proposition I.3.4]{qian2009smooth}, that is, $h_{\mu}(f) = h_{\tilde{\mu}}(\tilde{f})$. Thus, by the variational principle, the topological entropy of $(X_f, \tilde{f})$ and $(X,f)$ is the same.

An Anosov endomorphism is topologically conjugate to its linearization at the natural extension level \cite[Theorem 1.20]{przytycki1976anosov}, so their natural extensions have the same topological entropy and corresponding invariant Borel probability measures. Thus, by the invariance of entropy between a system and its natural extension, Anosov endomorphisms are intrinsically ergodic. We investigate this property for partially hyperbolic endomorphisms.

\subsection{Partially hyperbolic endomorphisms}

Partial hyperbolicity means that the dynamics presents directions with hyperbolicity dominating a central direction. On our context, we have $f: M \to M$ a local diffeomorphism, and we do not have a global invariant splitting in general. Indeed, even if $E^c_x$ is trivial ($f$ is an Anosov endomorphism), the case in which there is an invariant splitting --- which we call \textit{special} --- is not robust \cite{przytycki1976anosov}. So the definition cannot be made by using a global invariant splitting. We can define partial hyperbolicity for splittings along given $f$-orbits or using invariant cones.

\begin{definition}
\label{def:abs-ph}
We say that a $C^1$ local diffeomorphism $f: M \to M$ is \textit{absolutely partially hyperbolic} if, for any $f$-orbit $\tilde{x} \in M_f$, we have a splitting $T_{x_i}M = E^u_{x_i} \oplus E^c_{x_i} \oplus E^s_{x_i}$, $i \in \mathbb{Z}$, such that
\begin{enumerate}
    \item it is $Df$-invariant: $Df_{x_i}E^\sigma_{x_i} = E^\sigma_{x_{i+1}}$, $\sigma \in \{u, c, s\}$, for all $i \in \mathbb{Z}$;
    \item if $\tilde{x}, \tilde{y}, \tilde{z} \in M_f$, and $v^s \in E^s_{x_i}$, $v^c \in E^c_{y_j}$ and $v^u \in E^u_{z_k}$ are unit vectors, $i, j, k \in \mathbb{Z}$, then
$$\Vert Df_{x_i} v^s \Vert < \Vert Df_{y_j} v^c \Vert < \Vert Df_{z_k} v^u \Vert.$$
Moreover, $\Vert \restr{Df}{E^s_{x_i}} \Vert < 1$, $\Vert \restr{Df}{E^u_{x_i}} \Vert > 1$ for any $\tilde{x} \in M_f$ and $i \in \mathbb{Z}$.
\end{enumerate}
\end{definition}

We say that $f$ is \textit{(weakly) partially hyperbolic} if the inequality $\Vert Df_{x} v^s \Vert < \Vert Df_{x} v^c \Vert< \Vert Df_{x} v^u \Vert$ on the above definition holds for unit $v^\sigma \in E^\sigma_x$, $\sigma \in \{u, c, s\}$, that is, we can make this comparison just for vectors on the same tangent space.

For endomorphisms, we may have that: $E^c$ is trivial, obtaining an Anosov endomorphism with contracting direction; $E^s$ is trivial; or $E^s$ and $E^u$ are both trivial, obtaining an expanding map. Absolutely partially hyperbolic endomorphisms then generalize Anosov endomorphisms (which in turn generalize expanding maps and Anosov diffeomorphisms) and absolutely partially hyperbolic diffeomorphisms (the invertible case).

To see that the class of partially hyperbolic endomorphisms is open in $C^1(M, M)$, for instance, it is convenient to work with an alternative definition. For this, we consider a cone family as $\mathcal{C} = \{\mathcal{C}(x)\}_{x \in M} \subseteq TM$ with each $\mathcal{C}(x) \subseteq T_xM$ being a cone. We say that $\mathcal{C}$ is \textit{$Df$-invariant} if $Df_x(\mathcal{C}(x)) \subseteq \textrm{Int } \mathcal{C}(f(x))$ and \textit{$Df^{-1}$-invariant} if $Df^{-1}_{f(x)}(\mathcal{C}(f(x)) \subseteq \textrm{Int } \mathcal{C}(x)$, where $Df^{-1}_{f(x)}: T_{f(x)}M \to T_xM$.

\begin{definition}
We say that a $C^1$ local diffeomorphism $f: M \to M$ is \textit{absolutely partially hyperbolic} if there are $Df$-invariant cone families $\mathcal{C}^u$ and  $\mathcal{C}^{uc}$, $Df^{-1}$-invariant cone families $\mathcal{C}^{cs}$ and  $\mathcal{C}^{s}$, and constants $C>1$, $0 < \lambda < \gamma_1 < \gamma_2 < \mu$, with $\lambda < 1 < \mu$, such that 
\begin{align*}
    \Vert Df_xv \Vert &> \mu \Vert v \Vert \mbox{ for all } v \in\mathcal{C}^u(x);\\
    \Vert Df_xv \Vert &> \gamma_2 \Vert v \Vert \mbox{ for all } v \in\mathcal{C}^{uc}(x);\\
    \Vert Df^{-1}_xv \Vert &> \gamma_1^{-1} \Vert v \Vert \mbox{ for all } v \in\mathcal{C}^{cs}(x);\\
    \Vert Df^{-1}_xv \Vert &> \lambda^{-1} \Vert v \Vert \mbox{ for all } v \in\mathcal{C}^s(x).\\
\end{align*}
\end{definition}

If $M$ is an orientable surface, then the existence of $f: M \to M$ a partially hyperbolic endomorphism implies that $M = \mathbb{T}^2$, since there is a $Df$-invariant line field $E^c \subseteq TM$ transversal to the cone field \cite{hall2021partially}. This direction is described generically with $E^c$, but it could also be contracting or expanding, with less expansion than the cones. A classification of partially hyperbolic endomorphisms on surfaces is given in \cite{hall2022classification}.

We recall that there are local and global unstable/stable manifolds for $f$ in $M$ \cite[Theorem 2.1]{przytycki1976anosov}, with the unstable manifolds depending on the whole orbit $\tilde{x}$. In general, a partially hyperbolic endomorphism $f: M \to M$ does not have a global invariant unstable or central bundle, and these manifolds do not form a foliation. However, if each point has only one unstable/central direction, we say that $f$ is \textit{u/c-special}, and we say that $f$ is \textit{special} if it is both u- and c-special, as it is the case for linear toral endomorphisms.

As the natural extension is an important tool to understand the ergodic properties of endomorphisms, we have that the lift of the endomorphism to the universal cover is the natural way to explore its geometric and differential properties, due to the following result.

\begin{proposition}[\cite{costa-micena2021}]
    Let $\overline{M}$ be the universal cover of $M$ and $F: \overline{M} \to \overline{M}$ a lift for $f$. Then $f$ is a partially hyperbolic endomorphism if and only if $F: \overline{M} \to \overline{M}$ is a partially hyperbolic diffeomorphism.
\end{proposition}

Thus, at the universal cover level, we do have global unstable and central bundles. This implies that there is an unstable foliation at the universal cover. If $F: \overline{M} \to \overline{M}$ is \textit{dynamically coherent}, then it also has a center foliation.

\begin{definition}
A partial hyperbolic diffeomorphism $f: M \to M$ is said to be \textit{dynamically coherent} if there are invariant foliations $W^{uc}$ ans $W^{cs}$ tangent to $E^{uc} = E^u \oplus E^c$ and $E^{cs} = E^c \oplus E^s$ respectively.
\end{definition}

Since partial hyperbolic endomorphisms do not have unstable or central foliations in general, unless they are special, dynamical coherence is defined as follows.

\begin{definition}
\label{def:dyn-cohe}
A partial hyperbolic endomorphism $f: M \to M$ is said to be \textit{dynamically coherent} if there are unique invariant leaves $W^{uc}(\tilde{x})$ and $W^{cs}(\tilde{x})$ tangent to $E^{uc}(\tilde{x}) = E^u(\tilde{x}) \oplus E^c(\tilde{x})$ and $E^{cs}(x) = E^c(\tilde{x}) \oplus E^s(x)$ respectively.
\end{definition}

Since $E^u(\tilde{x})$ has the greatest expansion rate, the complementary direction $E^{cs}(x)$ is uniquely defined, not depending on the past orbit of $x$ (see \cite[Lemma 2.5]{costa-micena2021}). This bundle can still be not uniquely integrable. If $f$ is a dynamically coherent partial hyperbolic endomorphism, its lift $F: \overline{M} \to \overline{M}$ is dynamically coherent.

In the case that $M = \mathbb{T}^n$, another property that we verify for the lift $F$ is a general geometric property for foliations on $\mathbb{R}^n$, called \textit{quasi-isometry}. It means that the metric given by the induced distance between to points along a leaf of the foliation is equivalent to the Euclidean distance between them.

\begin{definition}
    Given a foliation $\mathcal{F}$ of $\mathbb{R}^n$, with $d_\mathcal{F}$ the distance along the leaves, we say that $\mathcal{F}$ is \emph{quasi-isometric} if there are constants $a, b > 0$ such that, for every $y \in \mathcal{F}(x)$,
    $$d_\mathcal{F}(x,y) \leq a \Vert x-y \Vert +b.$$
\end{definition}

If the foliation $\mathcal{F}$ is uniformly continuous, we can take $b=0$ in the above definition.
    
For special Anosov endomorphisms, quasi-isometry is guaranteed by the existence and the properties of a conjugacy between the map $f$ and a linear one, and we cover such conjugacy in the next subsection. For systems with partial hyperbolicity, quasi-isometry of the stable and unstable leaves implies dynamical coherence \cite{brin2003dynamical}. This geometric property is important as well to our uniqueness result.

\subsection{Hyperbolic linearization and semiconjugacy}
\label{subsec:semicon}
Let $f: \mathbb{T}^n \to \mathbb{T}^n$ be an endomorphism and $A$ its linearization, that is, $A$ is the unique linear map that induces the same homomorphism of $\mathbb{Z}^n \cong \pi_1(\mathbb{T}^n)$ as $f$. Consider $F: \mathbb{R}^n \to \mathbb{R}^n$ a lift of $f$ and $A: \mathbb{R}^n \to \mathbb{R}^n$ the linear lift of $A$ to $\mathbb{R}^n$. If $A$ is hyperbolic, then by \cite[Theorem 8.2.1]{aoki1994topological} and its proof, there is a unique continuous surjection $H: \mathbb{R}^n \to \mathbb{R}^n$ on the universal cover with 
    \begin{itemize}
        \item $A \circ H = H \circ F$;
        \item $d(H, Id) < K$ and $K$ goes to $0$ as $d_ {C^1}(f,A)$ tends to $0$;
        \item $H$ is uniformly continuous.
    \end{itemize}

This semiconjugacy is not necessarily preserved under deck transformations, that is, it does not necessarily projects to a semiconjugacy in $\mathbb{T}^n$. If $f$ is an Anosov endomorphism ($E^c$ is trivial), $H$ projects to $\mathbb{T}^n$ if and only if $f$ is special \cite[Proposition 7]{cantarino2021anosov}. But we can use $H$ to induce a semiconjugacy $\tilde{h}: \mathbb{T}^n_f \to \mathbb{T}^n_A$ between the natural extensions of $f$ and $A$. This is a consequence of \cite[Propositions 7.2.4 and 8.3.1]{aoki1994topological}, which we describe briefly.

Firstly, we need to understand the structure of the natural extension of a toral covering map $f: \mathbb{T}^n \to \mathbb{T}^n$ as a topological group. For a specific finite covering $\tilde{\mathbb{T}^n}$, the natural extension $(S, \tilde{F}) = \varprojlim(\tilde{\mathbb{T}^n}, F')$ of the lift $F'$ is constructed and proved to be a \emph{solenoidal group}, that is, a compact connected abelian group with finite topological dimension \cite[\S 7.2]{aoki1994topological}, that obeys the following the commutative diagram.

    \[
    \begin{tikzcd}[column sep=5em, row sep=2.5em]
    \mathbb{R}^n \arrow[rr,"F" near start] \arrow[dr,<->,swap,"\psi"] \arrow[dd,swap,"p''" near start] &&
        \mathbb{R}^n \arrow[dd,swap,"p''" near start] \arrow[dr,<->,"\psi"'] \\
    & \stackrel[\subseteq \; \mathbb{R}^p \oplus S_q]{}{\psi(\mathbb{R}^n)} \arrow[rr,crossing over,"\overline{F}" near start] \arrow[rr,crossing over,"\overline{F}"' near start,  shift right=1.5ex] &&
      \stackrel[\subseteq \; \mathbb{R}^p \oplus S_q]{}{\psi(\mathbb{R}^n)} \arrow[dd,"p_1" near start] \\
    \tilde{\mathbb{T}^n} \arrow[rr,"F'" near start] \arrow[dr,swap,"\varprojlim"] \arrow[dd,"p'"' near start] && \tilde{\mathbb{T}^n} \arrow[dr,swap,"\varprojlim"] \arrow[dd,"p'"' near start]\\
    & S := \tilde{\mathbb{T}^p} \oplus S_q \arrow[rr, crossing over, "\tilde{F}" near start] \arrow[uu,<-,crossing over,"p_1" near end] && \tilde{\mathbb{T}^p} \oplus S_q \arrow[dd,<->,"\beta" near end]\\
    \mathbb{T}^n \arrow[rr,"f" near start] \arrow[dr,swap,"\varprojlim"] && \mathbb{T}^n \arrow[dr,swap,"\varprojlim"]\\
    & \mathbb{T}^n_f \arrow[rr,"\tilde{f}" near start] \arrow[uu,<->,crossing over,"\beta" near start] && \mathbb{T}^n_f\\
    \end{tikzcd}
    \]
    
The lift of $f$ and $F'$ to the universal cover, $F: \mathbb{R}^n \to \mathbb{R}^n$ is homeomorphic to its image under an injective function $\psi$. The image $\psi(\mathbb{R}^n)$ is dense on $\mathbb{R}^p \oplus S_q$, where $p+q = n$ and $S_q$ is a solenoidal group. Thus $\overline{F} = \psi^{-1} \circ F \circ \psi$ can be extended to $\mathbb{R}^p \oplus S_q$.

The map $p_1: \mathbb{R}^p \oplus S_q \to \tilde{\mathbb{T}^p} \oplus S_q$ is a projection and $\overline{F}$ can be projected to $\tilde{F}$. Then \cite[Theorem 7.2.4]{aoki1994topological} gives us that $(S, \tilde{F}) = \varprojlim(\tilde{\mathbb{T}^n}, F')$. Finally, we have an isomorphism $\beta$ between $(\mathbb{T}^n_f, \tilde{f})$ and $(S, \tilde{F})$ \cite[Lemma 7.2.5]{aoki1994topological}.

The same constructions can be made for the linearization $A: \mathbb{T}^n \to \mathbb{T}^n$. If $A$ is hyperbolic, the semiconjugacy $H: \mathbb{R}^n \to \mathbb{R}^n$ on the universal cover is carried to $\mathbb{R}^p \oplus S_q$ as $\overline{H} = \psi^{-1} \circ H \circ \psi$, and $\overline{H}$ is shown in \cite[Theorem 8.3.1]{aoki1994topological} to project to $S$. Therefore, there is a semiconjugacy $\tilde{h}: \mathbb{T}^n_f \to \mathbb{T}^n_A$ between $(\mathbb{T}^n_f, \tilde{f})$ and $(\mathbb{T}^n_A, \tilde{A})$.

The existence of this semiconjugacy has consequences to the metric and topological entropies of $f$. Indeed, if $f:X\to X$ and $g:Y\to Y$ are continuous maps, there is $\phi:X\to Y$ a continuous surjection such that $\phi \circ f=g\circ \phi,$ and the measures $\mu$ and $\nu = h_*\mu$ are $f$ and $g$-invariant, respectively, then
\begin{enumerate}
    \item $h_{\nu}(g)\leq h_{\mu}(f)$ and
    \item $h_{top}(g)\leq h_{top}(f).$
\end{enumerate}

%%%    TEO A   %%%
\section{Proof of Theorem \ref{teoA}}
\label{sec:teoA}

The main step to prove Theorem \ref{teoA} is in demonstrating that the inverse limit $\tilde{f}: M_f \to M_f$ of $f$ is h-expansive. The proof is similar to the one for the invertible case, as proven by \cite{cowiesonyoung}, and we refer to \cite{diazFisher} for a presentation closer to ours.

\begin{theorem}
    If $f: M \to M$ is a $C^1$ partially hyperbolic endomorphism with one-dimensional center bundle, then $\tilde{f}: M_f \to M_f$ is $h$-expansive.
\end{theorem}

\begin{proof}
    In order to prove that the topological entropy of $\tilde{f}$ restricted to $\Gamma_\varepsilon(\tilde{x})$ is equal to $0$ for each $\tilde{x} \in M_f$, we show that $\Gamma_\varepsilon(\tilde{x})$ is contained on a local center-stable disk and the exponential growth of its spanning sets is given on a local center curve. This curve is one-dimensional, and its iterates have bounded length, so the entropy along it is $0$.
    
    Since $M$ is a closed manifold, we have that the natural extension is a fiber bundle $(\tilde{X}, X, \pi, C)$, where the fiber $C$ is a Cantor set \cite[Theorem 6.5.1]{aoki1994topological}.
    
    Let $\beta > 0$ be sufficiently small such that
    \begin{itemize}
        \item for each $x \in M$, $\pi^{-1}(B(x, \beta)) \simeq B(x, \beta) \times \pi^{-1}(\{x\})$;
        \item there is $\delta \in (0, \beta)$ such that, for each $\tilde{x} \in M_f$ and $\tilde{y} \in \tilde{B}(\tilde{x}, \beta) \simeq B(x, \beta) \times \{\tilde{x}\}$, if $d(\tilde{x}, \tilde{y}) < \delta$ then
        \begin{itemize}
            \item $W^s_\beta(x) \pitchfork \gamma^c_\beta(\tilde{y})$ is a singleton, where $\gamma^c_\beta(\tilde{y})$ is a curve with center $y = \pi(\tilde{y})$ and radius lesser than $\beta$ that is tangent to the central bundle;
            \item $W^u_\beta(\tilde{x}) \pitchfork D^{cs}_\beta(\tilde{y})$ is a singleton, where
            $$D^{cs}_\beta(\tilde{y}) = \bigcup_{z \in \gamma^c_\beta(\tilde{y})} W^s_\beta(z),$$
            is a disk tangent to $E^c \oplus E^s$.
        \end{itemize}
    \end{itemize}

\begin{figure}[ht]
   \centering
    \def\svgwidth{.42\linewidth}
        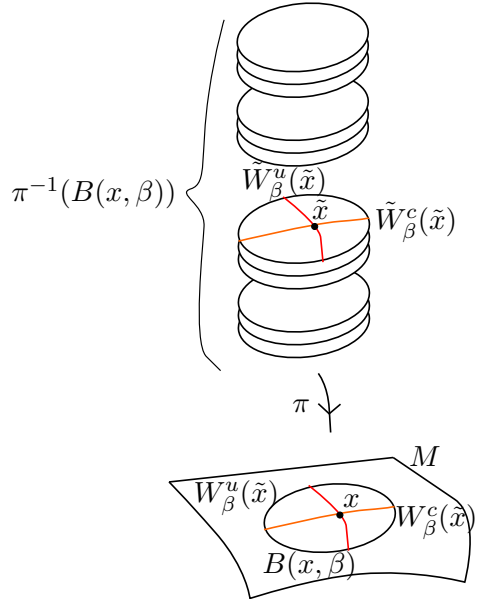
        \caption{$\pi^{-1}(B(x, \beta)) \simeq B(x, \beta) \times \pi^{-1}(\{x\})$ represented for a two-dimensional case.}
        \label{fig:mme1}
\end{figure}
    
    In other words, we are taking $\beta$ small enough to have inside the ball $B(x, \beta)$ a local product structure that depends only on $\tilde{x}$, so that it can be seen either in $B(x, \beta) \subseteq M$ or in $\tilde{B}(\tilde{x}, \beta)$, the connected component of $\tilde{x}$ in $\pi^{-1}(B(x, \beta)) \subseteq M_f$. See Figure \ref{fig:mme1}. Note that we are not assuming dynamical coherence, since we only need the existence of center curves, which is given by the fact that $\dim E^c = 1$.
    
    Let $\alpha > 0$ be such that $\lambda \alpha < \delta$, where $\lambda = \max_{x \in M} \Vert \restr{Df}{E^u} \Vert$. Then we can construct foliated boxes $V(\tilde{x}) \subseteq B(x, \beta)$ with uniform size $\alpha$ (not depending on $x \in M$) such that they have local product structure and their images $f(V(\tilde{x})) \subseteq B(f(x), \beta)$ also have local product structure. More precisely,
    $$V(\tilde{x}) := \bigcup_{y \in D^{cs}_\alpha(\tilde{x})} W^u_\alpha(\tilde{y}),$$
    where $\tilde{y} = \pi^{-1}(\{y\}) \cap \tilde{B}(\tilde{x}, \beta).$
    
    By construction, these boxes are small enough so that we can lift then to the natural extension inside $\tilde{B}(\tilde{x}, \beta)$, and we denote this lift by $\tilde{V}(\tilde{x}) := \pi^{-1}(V(\tilde{x})) \cap \tilde{B}(\tilde{x}, \beta)$, see Figure \ref{fig:mme2}. We use $\tilde{A}$ to denote the lift of any subset $A$ of $V(\tilde{x})$ to $\tilde{V}(\tilde{x})$. Then we can use the inverse $\tilde{f}^{-1}$ on each box, and the proof is concluded exactly as in \cite[Theorem 1.2]{diazFisher}, which we include here for completeness.

    \begin{figure}[ht]
        \centering
        \def\svgwidth{.45\linewidth}
        %% Creator: Inkscape 1.1.2 (b8e25be833, 2022-02-05), www.inkscape.org
%% Accompanies image file 'mme2.pdf' (pdf, eps, ps)
\begingroup%
  \makeatletter%
  \providecommand\color[2][]{%
    \errmessage{(Inkscape) Color is used for the text in Inkscape, but the package 'color.sty' is not loaded}%
    \renewcommand\color[2][]{}%
  }%
  \providecommand\transparent[1]{%
    \errmessage{(Inkscape) Transparency is used (non-zero) for the text in Inkscape, but the package 'transparent.sty' is not loaded}%
    \renewcommand\transparent[1]{}%
  }%
  \providecommand\rotatebox[2]{#2}%
  \newcommand*\fsize{\dimexpr\f@size pt\relax}%
  \newcommand*\lineheight[1]{\fontsize{\fsize}{#1\fsize}\selectfont}%
  \ifx\svgwidth\undefined%
    \setlength{\unitlength}{283.46456693bp}%
    \ifx\svgscale\undefined%
      \relax%
    \else%
      \setlength{\unitlength}{\unitlength * \real{\svgscale}}%
    \fi%
  \else%
    \setlength{\unitlength}{\svgwidth}%
  \fi%
  \global\let\svgwidth\undefined%
  \global\let\svgscale\undefined%
  \makeatother%
  \begin{picture}(1,0.8)%
    \lineheight{1}%
    \setlength\tabcolsep{0pt}%
    \put(0,0){\includegraphics[width=\unitlength,page=1]{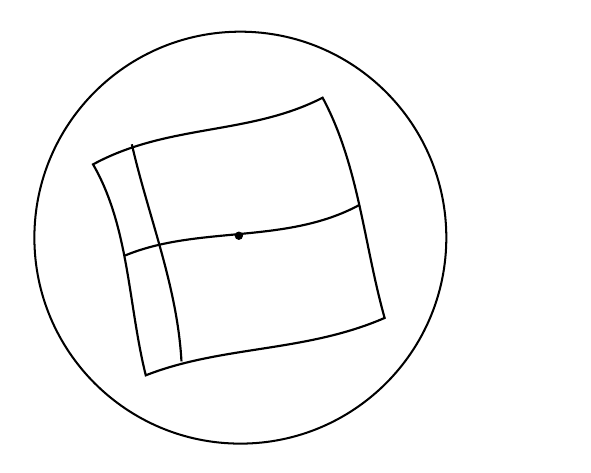}}%
    \put(0.56463799,0.71858708){\makebox(0,0)[lt]{\lineheight{1.25}\smash{\begin{tabular}[t]{l}$\tilde{B}(\tilde{x}, \beta)$\end{tabular}}}}%
    \put(0.73008995,0.57414482){\makebox(0,0)[lt]{\lineheight{1.25}\smash{\begin{tabular}[t]{l}$\tilde{V}(\tilde{x})$\end{tabular}}}}%
    \put(0.40181213,0.34825039){\makebox(0,0)[lt]{\lineheight{1.25}\smash{\begin{tabular}[t]{l}$\tilde{x}$\end{tabular}}}}%
    \put(0.16898026,0.59687658){\makebox(0,0)[lt]{\lineheight{1.25}\smash{\begin{tabular}[t]{l}$\tilde{W}^u_\alpha(\tilde{y})$\end{tabular}}}}%
    \put(0,0){\includegraphics[width=\unitlength,page=2]{mme2.pdf}}%
    \put(0.28365566,0.33597214){\makebox(0,0)[lt]{\lineheight{1.25}\smash{\begin{tabular}[t]{l}$\tilde{y}$\end{tabular}}}}%
    \put(0.61978867,0.4454599){\makebox(0,0)[lt]{\lineheight{1.25}\smash{\begin{tabular}[t]{l}$\tilde{\gamma}(\tilde{x})$\end{tabular}}}}%
    \put(0,0){\includegraphics[width=\unitlength,page=3]{mme2.pdf}}%
  \end{picture}%
\endgroup%

        \caption{The box $\tilde{V}(\tilde{x})$ contained in the connected component of $\tilde{x}$ on $\pi^{-1}(B(x, \beta))$, represented here for the two-dimensional case.}
        \label{fig:mme2}
    \end{figure}
    
    Since the sizes of the boxes $\tilde{V}(\tilde{x})$ are uniform, there is $\varepsilon > 0$ such that $B(\tilde{x}, \varepsilon) \cap \tilde{B}(\tilde{x}, \beta) \subseteq \tilde{V}(\tilde{x})$ for any $\tilde{x} \in M_f$.  
    
    Fixed $\tilde{x} \in M_f$, we consider $\tilde{x}^n : = \tilde{f}^n(\tilde{x})$ and $\gamma_n := \gamma^c_\alpha(\tilde{x}^n)$ chosen in such way that $f(\gamma_{n-1}) \cap \gamma_n$ contains a open interval around $x_n = \pi(\tilde{x}^n)$. This is possible since the central direction is invariant under $f$. So we are fixing each central curve coherently along the forward orbit of $\tilde{x}$, and each foliated box $\tilde{V}$ henceforth is with respect to this choice.
    
    \begin{proof}[Claim 1:] $\Gamma_\varepsilon(\tilde{x}) \subseteq D^{cs}_\alpha(\tilde{x})$.
        
        Indeed, for each $\tilde{z} \in \Gamma_\varepsilon(\tilde{x})$, we have that $\tilde{z}^k := \tilde{f}^k(\tilde{z}) \in B(\tilde{f}^k(\tilde{x}), \varepsilon)$ for each $k \in \mathbb{Z}$ by the definition of $\Gamma_\varepsilon(\tilde{x})$, thus $\tilde{z}^k \in \tilde{V}(\tilde{x}^k)$.
        
        Consider the projections
        $$\begin{array}{cccc}
          p^{cs}:  & \tilde{V}(\tilde{x}) & \to & \tilde{D}^{cs}_\alpha(\tilde{x}) \\
          & \tilde{z} & \mapsto & \tilde{W}^u_\alpha(\tilde{z}) \cap \tilde{D}^{cs}_\alpha(\tilde{x})
        \end{array}$$
        and
        $$\begin{array}{cccc}
          p^{c}:  & \tilde{D}^{cs}_\alpha(\tilde{x}) & \to & \tilde{\gamma}(\tilde{x}) \\
          & \tilde{y} & \mapsto & \tilde{W}^s_\alpha(\tilde{y}) \cap \tilde{\gamma}(\tilde{x}),
        \end{array}$$
        that are well defined by the local product structure. Define $\tilde{y}^n := p^{cs}(\tilde{z}^n)$ and $\tilde{w}^n := p^{c}(\tilde{y}^n)$ for each $n \in \mathbb{N}$.
        
        Suppose that $\tilde{z}^n \notin \tilde{D}^{cs}_\alpha(\tilde{x}^n)$, that is, $\tilde{z}^n \neq \tilde{y}^n$. Then, since $\tilde{y}^n \in \tilde{W}^u(\tilde{z}^n)$, there is $m \in \mathbb{N}$ such that $d(\tilde{z}^{n+m}, \tilde{y}^{n+m}) > \alpha$, which implies that $\tilde{z}^{n+m} \notin \tilde{V}(\tilde{x}^{n+m})$, a contradiction.
    \end{proof}
    
    Define $$\Gamma^c := \bigcap_{n \geq 0} \tilde{f}^{-n}(\tilde{\gamma_n}),$$ as the set of points in $\tilde{\gamma}(\tilde{x})$ such that their images under $\tilde{f}^n$ are in $\tilde{\gamma_n}$. The length of $\tilde{f}^n(\Gamma^c)$ is lesser than $2 \alpha$ for all $n \in \mathbb{N}$. In particular we have that $h(\tilde{f}, \Gamma^c) = 0$, see \cite[Lemma 4.2]{buzzifishersambarinovasquez2012}, for instance.
    
    \begin{proof}[Claim 2:] $h(\tilde{f}, \tilde{D}^{cs}_\alpha(\tilde{x})) = h(\tilde{f}, \Gamma^c)$.
    
    Indeed, if $S$ is an $(n, \epsilon/2)$-spanning set for $\Gamma^c$, then it is an $(n, \epsilon)$-spanning set for
    $$B(\Gamma^c, \epsilon/2) = \{\tilde{x} \in M_f \; : \; \exists \; \tilde{y} \in \Gamma^c \mbox{ with } d(\tilde{x}, \tilde{y}) < \epsilon/2 \}.$$
    
    If $n \in \mathbb{N}$ is sufficiently large, then $\tilde{f}^n(\tilde{D}^{cs}_\alpha(\tilde{x})) \subseteq B(\tilde{f}^n(\Gamma^c), \epsilon/2)$, and $S$ is an $(n, \epsilon)$-spanning set for $\tilde{D}^{cs}_\alpha(\tilde{x})$. Thus
    \begin{align*}
        h(\tilde{f}, \Gamma^c) &= \displaystyle\lim_{\epsilon \rightarrow 0} \displaystyle\limsup_{n\rightarrow \infty}\frac{1}{n}\log \min\{\# S: S \subseteq \Gamma^c \ {\rm is  \ an \ }(n,\epsilon){\rm -spanning \ set} \}\\
        &= \displaystyle\lim_{\epsilon \rightarrow 0} \displaystyle\limsup_{n\rightarrow \infty}\frac{1}{n}\log \min\{\# S: S \subseteq \tilde{D}^{cs}_\alpha(\tilde{x}) \ {\rm is  \ an \ }(n,\epsilon/2){\rm -spanning \ set} \}\\
        &= h(\tilde{f}, \tilde{D}^{cs}_\alpha(\tilde{x})).
    \end{align*}

\end{proof}

Since $h(\tilde{f}, \Gamma^c) = 0$ and $\Gamma_\varepsilon(\tilde{x}) \subseteq \tilde{D}^{cs}_\alpha(\tilde{x})$, then Claim 2 implies that $h(\tilde{f}, \Gamma_\varepsilon(\tilde{x})) = 0$.
\end{proof}
 
The previous theorem implies that $\tilde{f}$ is $h$-expansive, and from Theorem \ref{emme} we have that $\tilde{\mu} \mapsto h_{\tilde{\mu}}(\tilde{f})$ is upper semicontinuous, where $\tilde{\mu}$ is an $\tilde{f}$-invariant measure. Remember that by Proposition \ref{prop:corresp-measures} there is a injective correspondence between $\tilde{f}$-invariant and $f$-invariant measures given by $\mu = \pi_*{\tilde{\mu}}$.

Now consider $\phi: M \to R$ a continuous potential on $M$. Consider the function $\mu \mapsto P_\mu(f, \phi)$, where 

$$P_\mu(f, \phi) := h_\mu(f) + \int_M \phi d\mu.$$

Then, since $h_\mu(f) = h_{\tilde{\mu}}(\tilde{f})$ and

$$\int_M \phi d\mu = \int_M \phi d\pi_*{\tilde{\mu}} = \int_{M_f} \phi \circ \pi d\tilde{\mu},$$
thus

$$P_\mu(f, \phi) = P_{\tilde{\mu}}(\tilde{f}, \phi \circ \pi) := h_{\tilde{\mu}}(\tilde{f}) + \int_{M_f} \phi \circ \pi d\tilde{\mu}.$$

So, for any potential given by $\phi \circ \pi$ (a lift to $M_f$ of a continuous potential on $M$), the function $\tilde{\mu} \mapsto P_{\tilde{\mu}}(\tilde{f}, \phi \circ \pi)$ is upper semicontinuous and, thus, admits a maximum, an equilibrium state. Additionally, this maximum $\tilde{\mu}$ projects to $\mu$, an equilibrium state for $\phi$.

\section{Proof of Theorem \ref{teoB}}
\label{sec:teoB}

 %The proof is similar to the one of Theorem \ref{teoB}, but the 
%For this result, the quasi-isometry and the semiconjugacy with a linear model on the ambient manifold allows a global argument.

Even without an unstable foliation for $f$ in $\mathbb{T}^n$, an unstable foliation does exist for $F$ in $\mathbb{R}^n$, and all computations in the universal cover work as in the invertible case. Thus, as in \cite[Lemma 3.4]{ures2012intrinsic}, we prove the following, where $H: \mathbb{R}^n \to \mathbb{R}^n$ is a lift for the semiconjugacy $h: \mathbb{T}^n \to \mathbb{T}^n$.

\begin{lemma}
    \label{lem:lemma1}
    If $\dim(E_F^c) = 1$, $f$ is dynamically coherent and $A = f_*$ is hyperbolic, then $A$ admits a partially hyperbolic splitting $\mathbb{R}^n = E^u_A \oplus E^c_A \oplus E^s_A$ with $\dim(E_A^c) = 1$. Additionally, $H(W^c_F(x)) = W^c_A(H(x))$ for all $x \in \mathbb{R}^n$.
\end{lemma}

If $A$ is hyperbolic, we also have that there is $\alpha > 0$ such that, for all $x, y \in \mathbb{R}^n$, $H(x) = H(y)$ if and only if $d(F^k(x), F^k(y)) < \alpha$ for all $k \in \mathbb{Z}$. Indeed, if $H(x) = H(y)$, then
$$A^k(H(x)) = A^k(H(y)) \implies H(F^k(x)) = H(F^k(y)) \; \mbox{ for all } k \in \mathbb{Z}.$$
Since $d(H(z), z) < K$ for all $z \in \mathbb{R}^n$, then
$$d(F^k(x), F^k(y)) \leq d(F^k(x), H \circ F^k(x)) + d(H \circ F^k(x), H\circ F^k(y)) + d(H \circ F^k(y), F^k(y)),$$
which is lesser than $2K$ for all $k \in \mathbb{Z}$.

For the reciprocal implication, we use the expansiveness of $A$ and the fact that it is linear to guarantee that it is expansive with any expansiveness constant (see, for instance, \cite[Lemma 8.2.3]{aoki1994topological}). Since
\begin{align*}
    &d(A^k \circ H(x), A^k \circ H(y)) = d(H \circ F^k(x), H \circ F^k(y)) \leq d(H \circ F^k(x), F^k(x))\\
    &+ d(F^k(x), F^k(y)) + d(F^k(y), H \circ F^k(y)) < 2K + \alpha
\end{align*}
for all $k \in \mathbb{Z}$, then $H(x) = H(y)$.

\begin{lemma}
    \label{lem:lemma2}
    If $A$ is hyperbolic, $F$ is dynamically coherent and $W^{u/s}_F$ is quasi-isometric, then $H(x) = H(y)$ implies that $y \in W^c_F(x)$.
\end{lemma}

\begin{proof}
Suppose that $y \notin W^c_F(x) = W^{uc}_F(x) \cap W^{cs}_F(x)$. Then $y \notin W^{cs}_F(x)$ or $y \notin W^{uc}_F(x)$. Let us see that the first case is absurd, and the second case is analogous.

If $y \notin W^{cs}_F(x)$, then there is $z \neq y$ such that $z = W^{cs}_F(x) \cap W^u_F(y)$. Indeed, the global product structure follows from the quasi-isometry, as proved in \cite[Theorem 1.1]{hammerlindl2012dynamics} for partially hyperbolic diffeomorphisms on (not necessary compact) manifolds. It then applies to $F: \mathbb{R}^n \to \mathbb{R}^n$.

Consider $D_{cs} = d_{cs}(x,z)$ and $D_u = d_u(y,z)$. There are $1 < \lambda_c < \lambda_u$ such that
\begin{align*}
    &d(F^k(x), F^k(z)) \leq \lambda_c^k D_{cs} \mbox{ and}\\
    &d_u(F^k(y), F^k(z)) \geq \lambda_u^k D_{u}
\end{align*}
for all $k \in \mathbb{Z}$. Since $W^u_F$ is quasi-isometric, we have $d_u(F^k(y), F^k(z)) \leq a d(F^k(y), F^k(z)) + b$ and thus 
$$d(F^k(y), F^k(z)) \geq \dfrac{\lambda_u^k D_u -b}{a}.$$

Therefore
\begin{align*}
    d(F^k(x), F^k(y)) &\geq d(F^k(y), F^k(z)) - d(F^k(x), F^k(z))\\
    &> \dfrac{\lambda_u^k D_u -b}{a} - \lambda_c^k D_{cs} \xrightarrow[]{n \to \infty} \infty,
\end{align*}
which implies that $H(x) \neq H(y)$, a contradiction.
\end{proof}

\begin{remark}
The above lemma requires that $n = \dim(\mathbb{T}^n) \geq 3$, for we need $\dim(E^\sigma_F) \neq 0$, $\sigma \in \{u, c, s\}$. If $n = 2$, we have the same result just by requiring that $A$ is hyperbolic. Indeed, in this case, $\dim(E^c) = \dim(E^u) = 1$ and we have dynamical coherence, quasi-isometry and global product structure by \cite{hall2021partially}. Supposing that $y \notin W^c_F(x)$, there is a unique $z \in W^u_F(y) \cap W^c_F(x)$, $z \neq y$, and the proof by absurd is the same.
\end{remark}

\begin{remark}
By lemmas \ref{lem:lemma1} and \ref{lem:lemma2}, for all $z \in \mathbb{R}^n$ we have that $H^{-1}(W^c_A(z)) = W^c_F(x)$ for any $x \in H^{-1}(z)$.
\end{remark}

\begin{lemma}
    \label{prop:prop1}
    If $A$ is hyperbolic, $\dim(E^c_F) = 1$, $F$ is dynamically coherent and $W^\sigma_F$ is quasi-isometric for $\sigma \in \{u, c, s\}$, then for all $z \in \mathbb{R}^n$ we have that $H^{-1}(z)$ is a compact and connected subset of $W^c_F(x)$ for $x \in H^{-1}(z)$.
\end{lemma}

\begin{proof}
We have that $H^{-1}(z)$ is closed, and it is bounded since $d(H, Id) < K$. For $x \in H^{-1}(z)$, $H^{-1}(z)$ it is a subset of $W^c_F(x)$ by Lemma \ref{lem:lemma2}. It remains to prove connectedness. Fixing $x, y \in H^{-1}(z)$ and given $w \in [x, y]_c$, we have by the quasi-isometry of $W^c_F$ that
\begin{align*}
    d(F^k(x) F^n(w)) &\leq d_c(F^k(x) F^n(w)) \leq d_c(F^k(x) F^n(y)) \\
    &\leq a d(F^k(x) F^n(y)) + b \leq 2 a K + b,\\
\end{align*}
which implies that $z = H(x) = H(w)$, as previously shown using the expansiveness of $A$. Thus, $[x, y]_c \subseteq H^{-1}(z)$ for all $x, y \in H^{-1}(z)$, and $H^{-1}(z)$ is connected.
\end{proof}

Let $\overline{\Gamma} = \{z \in \mathbb{R}^n \; : \; \#H^{-1}(z) > 1 \}$ be the set of points for which $H$ fails to be invertible. Consider $p(\overline{\Gamma}) = \Gamma$, with $p: \mathbb{R}^n \to \mathbb{T}^n$ the canonical projection.

\begin{lemma}
    \label{lem:lemma3}
    Under the hypotheses of Lemma \ref{prop:prop1}, $m(\overline{\Gamma})$ = 0.
\end{lemma}

\begin{proof}
For all $z \in \mathbb{R}^n$ there is $x \in \mathbb{R}^n$ such that $H^{-1}(W^c_A(z)) = W^c_F(x)$. Consider $\overline{\Gamma}^c_z = W^c_A(z) \cap \overline{\Gamma}$. Then $\{H^{-1}(y) \; : \; y \in \overline{\Gamma}^c_z \}$ is a family of disjoint nontrivial intervals in $W^c_A(z)$, we obtain that $\overline{\Gamma}^c_z$ is countable for all $z \in \mathbb{R}^n$. Therefore, since $W^c_A$ is a linear foliation, Fubini's theorem provides that $m(\overline{\Gamma})=0$.
\end{proof}

Thus, $\Gamma = p(\overline{\Gamma})$ also has zero volume on $\mathbb{T}^n$, and it satisfies $\Gamma = \{z \in \mathbb{T}^n \; : \; \#h^{-1}(z) > 1 \}$ since $H$ is a lift for $h$. Moreover, $h^{-1}(\Gamma)=\Gamma.$

%\begin{lemma}{\cite[Lemma I.3.3]{qian2009smooth}}
 %   \label{lem:lemma4}
  %  Let $X$ and $Y$ be compact metric spaces and consider $T: X \to X$ and $S: Y \to Y$ measurable maps. If there is a continuous surjective map $h: X \to Y$ such that $S \circ h = h \circ T$, then, for any $S$-invariant Borel probability measure $\mu$ on $Y$, there is a $T$-invariant Borel probability measure $\nu$ on $X$ such that $h_*\nu = \mu$.
%\end{lemma}

%Therefore, for any $A$-invariant probability measure $\nu$ on $\mathbb{T}^n$, there is an $f$-invariant probability measure $\mu$ such that $h_*\nu = \mu$. In our case $\nu = m$ is the volume measure and $m(\Gamma) = 0$, which give us the uniqueness of such $\mu$. Indeed, if the set of points in which $h$ fails to be invertible has measure $0$, then $h$ is an ergodic equivalence, and $\mu$ is characterized by the values it assigns to sets $h^{-1}(U)$, with $U \subseteq \mathbb{T}^n$ an open set.

%It remains to prove that such $\mu$ is precisely the measure of maximal entropy for $f$, and for it we use the Ledrappier--Walters variational principle, it suffices to prove that $h(f, h^{-1}(z)) = 0$ for all $z \in \mathbb{T}^n_A$. Indeed, then we have that
%\begin{align*}
    %h_\mu(f) &= h_m(A) = h_\textrm{top}(A) \mbox{ and }\\
    %h_\eta(f) &\leq h_{h_*\eta}(A) < h_\textrm{top}(A),\\
%\end{align*} for any $\eta \neq \mu \in \mathcal{M}(f)$.

Lemma \ref{prop:prop1} implies that $H^{-1}(\overline{z})$ is a compact and connected one-dimensional central disk, with its length bounded with a constant that does not depend on $\overline{z}$. But if $\overline{x}, \overline{y} \in \mathbb{R}^n$ are such that $H(\overline{x}) = H(\overline{y})$, then we have that $d(F^k(\overline{x}), F^k(\overline{x})) < 2K$ for any $k \in \mathbb{Z}$. Thus $\overline{x} \in H^{-1}(\overline{z})$ if and only if $F^n(\overline{x}) \in H^{-1}(A^n(\overline{z}))$ for any $n \in \mathbb{N}$. Therefore, the length of $F^n(H^{-1}(\overline{z}))$ is also bounded with the same constant than $H^{-1}(\overline{z})$. These bounds are the same for the projections of the stable manifolds, and this implies $h(f, h^{-1}(z)) = 0$ for all $z \in \mathbb{T}^n$.

By Theorem \ref{teoA} there exists a measure of maximal entropy $\mu$. Then, using Ledrappier--Walters' formula, we have that $h_{\ast}\mu=m.$ Now, we want to prove that $\mu$ is the unique measure of maximal entropy that project into $m$. By contradiction, suppose that there exists a measure of maximal entropy $\eta\neq \mu$, then  $h_{\ast}\eta=m=h_{\ast}\mu$. From Lemma \ref{lem:lemma3}, it follows that $\mu(\Gamma)=\eta(\Gamma)=0$ and for every continuous function $\psi:\mathbb{T}^{n}\to \mathbb{R}$ we have that 
\begin{align*}
 \int \psi d\mu &= \int_{\mathbb{T}^n\setminus \Gamma}\psi d\mu = \int_{\mathbb{T}^n\setminus \Gamma}\psi\circ h^{-1}\circ h d\mu\\ &=\int_{\mathbb{T}^n\setminus\Gamma}\psi\circ h^{-1} dh_{\ast}\mu =\int_{\mathbb{T}^{n}\setminus\Gamma}\psi\circ h^{-1}dh_{\ast}\eta
 =\int \psi d\eta.   
\end{align*}
Hence, $\mu=\eta$, a contradiction. Therefore, there exists a unique measure of maximal entropy.

%\subsection*{Data availability} All data generated or analysed during this study are included in this published article [and its supplementary information files].

%\subsection*{Declarations} All the authors of the paper approve of this submission, and each of them has equally contributed to the findings of this paper. None of the authors have any potential conflicts of interest or competing interests that are relevant to the content of this article.

%%%    BIBLIOGRAFIA   %%%
%\nocite{*}
\bibliographystyle{acm}
\bibliography{references}
\end{document}